\pdfoutput=1
\documentclass[11pt]{article}

\usepackage{style}

\title{\color{DLtitle}\rmfamily\spacedallcaps{Operads without coalgebras}}
\author{\spacedlowsmallcaps{brice le grignou}\and\spacedlowsmallcaps{damien lejay}}
\date{} 

\tikzcdset{%
	arrow style=tikz,
	>/.tip={Stealth[length=3pt, width=5pt, inset=1.8pt]}
}
\hypersetup{%
	pdftitle={Operads without coalgebras},
	pdfauthor={\textcopyright%
		Brice Le Grignou and Damien Lejay},
	pdfsubject={},
	pdfkeywords={},
	pdfcreator={pdfLaTeX},
	pdfproducer={LaTeX with hyperref}
}
\newcommand{\tocconfig}{%
	\addtocontents{toc}{\protect\vspace{\beforebibskip}}
	\addcontentsline{toc}{section}{Acknowledgements}
	\addcontentsline{toc}{section}{\refname}
}
\newcommand{\IdentityMap}{\mathrm{id}}
\newcommand{\IdentityFunctor}{\mathrm{Id}}
\newcommand{\SymmetricGroup}[1]{\mathbf{S}_{#1}}
\newcommand{\Cosmos}{\mathfrak C}
\newcommand{\CosmosUnit}{\mathbf 1}
\newcommand{\Insane}{\mathsf{Ins}}
\newcommand{\CategoryOfPointedCogebrasOver}[1]{{#1}\text{-}\mathsf{cog}_\bullet}
\newcommand{\Interval}{\mathbf D^1}
\newcommand{\CategoryOfCogebrasOver}[1]{{#1}\text{-}\mathsf{cog}}
\newcommand{\CategoryOfAlgebrasOver}[1]{{#1}\text{-}\mathsf{alg}}
\newcommand{\Field}[1]{\boldsymbol #1}
\newcommand{\CofreeCogebraFunctor}[1]{\operatorname{L}^{#1}}
\newcommand{\Insertion}{\mathbin{\vartriangleleft}}
\newcommand{\CategoryOfComodulesOver}[1]{#1\text{-}\mathsf{comod}}
\newcommand{\CategoryOfModulesOver}[1]{#1\text{-}\mathsf{mod}}
\newcommand{\CofreeCogebraComonad}[1]{\mathbf{L}^{#1}}
\newcommand{\IsAdjointTo}{\mathrel{\dashv}}
\newcommand{\Forget}{\operatorname{Forget}}
\newcommand{\Free}{\operatorname{Free}}
\newcommand{\Cofree}{\operatorname{Cofree}}
\newcommand{\AdditiveCosmos}{\mathfrak A}
\newcommand{\CategoryOfChainComplexesIn}[1]{\operatorname{Ch} (#1)}
\newcommand{\InternalHom}[2]{\mleft[#1, #2\mright]}
\newcommand{\CoendomorphismOperad}[1]{\operatorname{Coend}_{#1}}
\newcommand{\EndomorphismOperad}[1]{\operatorname{End}_{#1}}
\newcommand{\Dual}{^\vee}
\newcommand{\UnitalAssociativeOperad}{\mathsf{uA}}
\newcommand{\RightInverse}{\sigma}
\newcommand{\CategoryOfVectorSpacesOver}[1]{\mathsf{Vect}_{#1}}

\begin{document}


\maketitle


\begin{abstract}
	We give an example of a non-trivial linear operad
	that only admits trivial coalgebras
	and give sufficient conditions
	ensuring that the cofree coalgebra functor be faithful.
\end{abstract}


\section{Introduction}

In the same way that operads are used to encode algebraic data, they
can equally be used for coalgebraic data. For example, the operad
governing associative algebras controls at the same time coassociative
coalgebras:
one operation \( \mu \) of arity \( 2 \) versus one cooperation
\( \delta \) of (co)-arity \( 2 \),
\[
	\mu \circ (\mu \otimes \IdentityMap)
	= \mu \circ (\IdentityMap \otimes \mu)
	\quad \text{versus}
	\quad (\delta \otimes \IdentityMap) \circ \delta
	= (\IdentityMap \otimes \delta) \circ \delta.
\]
As we shall see, the theory of coalgebras over an operad is not always
as gentle as that of algebras.

The word `operad' was carefully crafted by Peter May as a
portmanteau of `operations' and `monad'
\cite{doi:10.1090/conm/202/02588}.
Indeed, to every operad \( P \) in a cocomplete closed symmetric
monoidal
category \( (\Cosmos, \otimes, \CosmosUnit) \) is associated a monad
\( \widetilde P \) with underlying functor
\[
	X \longmapsto P \Insertion X
	\coloneqq \coprod_{n \in \Naturals} P (n)
	\otimes_{\SymmetricGroup n} X^{\otimes n}.
\]
The category of \( P \)-algebras is then the category of
\( \widetilde P \)-modules
\[
	\CategoryOfAlgebrasOver P
	\IsNaturallyIsomorphicTo \CategoryOfModulesOver {\widetilde P}.
\]
In many cases the assignment \( P \mapsto \widetilde P \) is faithful
(for example, operads in sets or operads in vector spaces
\cite{doi:10.1007/bfb0072514}).
In particular, \( P \) is the trivial operad if and only if
\( \widetilde P \) is the trivial monad. Monads arising from operads
are also well-behaved: when the unit \( \CosmosUnit \to P \) is
a split monomorphism, the free \( P \)-algebra functor is always
faithful, that is the unit of the adjunction
\( \Free \IsAdjointTo \Forget \)
\[
	X \longrightarrow P \Insertion X
\]
is a monomorphism for every \( X \in \Cosmos \).

This is no longer the case for coalgebras over \( P \). For one, the
forgetful functor
\( \CategoryOfCogebrasOver P \to \Cosmos \)
may not have a right adjoint, that is, there may not exist cofree
\( P \)\=/coalgebras. When it does, the category of \( P \)-coalgebras
is then equivalent to the category of comodules over a comonad
\[
	\CategoryOfCogebrasOver P
	\IsNaturallyIsomorphicTo \CategoryOfComodulesOver
	{\CofreeCogebraComonad P}
\]
but the functor \( P \mapsto \CofreeCogebraComonad P \)
may no longer be faithful. In what follows we give an example
of a non-trivial linear operad \( \Insane \) whose associated
comonad is the zero comonad
\[
	\Insane \neq 0 \qand \CofreeCogebraComonad \Insane = 0.
\]
In other terms, all \( \Insane \)-coalgebras are trivial.
In the last section, we give sufficient
conditions in the linear setting on operads \( P \) so that the
cofree coalgebra functor be faithful, i.e.\ so that the counit of
the comonad
\[
	\CofreeCogebraFunctor P X \longrightarrow X
\]
be an epimorphism for every \( X \).

\section{Coalgebras over an operad}

In a closed symmetric monoidal category with tensor product
\( \otimes \) and internal hom \( [-,-] \), a coalgebra over an operad
\( P \) is the data of an object \( V \) together with
a morphism of operads \( a \From P \to \CoendomorphismOperad V \)
where \( \CoendomorphismOperad V \) is the coendomorphism operad
of \( V \) given by
\[
	\CoendomorphismOperad V (n)
	\coloneqq \InternalHom V {V^{\otimes n}},
\]
with obvious right \( \SymmetricGroup n \)-action and compositions.
Similarly,
a \( P \)\=/coalgebra structure on \( V \) is the data of maps
\( P(n) \otimes V \to V^{\otimes n} \) suitably associative,
unital and equivariant.

In the cartesian case, the coendomorphism operad simplifies to
\[
	\CoendomorphismOperad V (n) = \InternalHom V {V^{n}}
	\IsNaturallyIsomorphicTo {\InternalHom V V}^n
	\IsNaturallyIsomorphicTo {\CoendomorphismOperad V (1)}^n.
\]
This trivialises the theory of \( P \)-coalgebras in cartesian
categories as is exemplified by the well-known fact that each set has
a unique coassociative counital coalgebra structure, the diagonal.
For this reason, we shall focus our attention on the additive case and
fix
a closed symmetric monoidal cocomplete additive category
\( (\AdditiveCosmos, \otimes, \CosmosUnit) \).
By a dg-operad, we shall mean an operad in the symmetric monoidal
category of chain complexes
\( (\CategoryOfChainComplexesIn \AdditiveCosmos,
\otimes, \CosmosUnit) \).

We shall also assume that cofree coalgebras exist:
for every morphism of dg-operads
\( f \From P \to P' \),
we require that the forgetful functor
\[
	\begin{tikzcd}
	\CategoryOfCogebrasOver {P'}
	\rar["f^\ast"]
		& \CategoryOfCogebrasOver P
	\end{tikzcd}
\]
admit a right adjoint \( f_\ast \).
In particular, for every dg-operad \( P \),
the unit morphism \( \CosmosUnit \to P \) yields the adjunction
\[
	\begin{tikzcd}[column sep = huge]
		\CategoryOfCogebrasOver P
		\arrow[r, shift left=2,"\Forget"]
			& \CategoryOfChainComplexesIn \AdditiveCosmos.
			\arrow[l, shift left=2, "\Cofree"]
	\end{tikzcd}
\]
This is the case for example when
\( \AdditiveCosmos \) is the category of vector spaces over a field,
where cofree coalgebras can be computed for every dg-operad
\cite{arXiv:1409.4688}.
Thanks to this assumption, the category of \( P \)-coalgebras becomes
comonadic over the category of chain complexes: the category of
\( P \)\=/coalgebras is equivalent to the category of comodules over
the comonad \( \CofreeCogebraComonad P \) coming from the
\( \Forget \IsAdjointTo \Cofree \) adjunction.
Its underlying functor is
\( \CofreeCogebraFunctor P = \Forget \circ \Cofree \).
In addition, for every morphism \( f \From P \to P' \), the
counit \( f^\ast f_\ast \to \IdentityFunctor \)
induces a morphism of comonads
(i.e.\ a morphism of comonoids in the category of endofunctors)
\[
	\CofreeCogebraComonad f
	\From \CofreeCogebraComonad {P'}
	\longrightarrow \CofreeCogebraComonad P
\]
on \( \CategoryOfChainComplexesIn \AdditiveCosmos \).
In other words we assume the existence of a
contravariant functor from the category of dg-operads to the category
of comonads.

\section{An operad without coalgebras}%
\label{sec:example_insane}

We shall build here an example of an operad without coalgebras
in the category of vector spaces over a field \( \Field K \)
by adding to the unital associative operad
\( \UnitalAssociativeOperad \)
an infinite number of operations of arity \( 0 \),
the idea being that a locally finite dimensional coalgebra
cannot support an infinite number of independent linear forms.

Let \( \Insane \) be the operad in
\( \CategoryOfVectorSpacesOver {\Field K} \)
whose algebras are the \( \UnitalAssociativeOperad \)-algebras
\( (\Lambda, \mu, \upsilon_0) \),
endowed with other arity zero operations
\( \upsilon_n \From \Field K \to \Lambda \)
for every \( n \geq 1 \) and an arity \( 1 \) operation \( I_\lambda \)
for every non-zero finitely supported sequence \( \lambda \)
of elements of \( \Field K \) with relation
\[
	\mu \circ \left(\sum_{n\in\Naturals}
	\lambda_n \upsilon_n \otimes \IdentityMap_\Lambda\right)
	\circ I_\lambda = \IdentityMap_\Lambda.
\]

The first thing to check is that \( \Insane \) is not a trivial operad.
For this one only needs to exhibit a non-trivial \( \Insane \)-algebra.
Let \( \Field L/\Field K \) be a field extension of infinite dimension
and let \( \upsilon_0 \coloneqq 1, \upsilon_1, \upsilon_2, \dots \) be a
countable family of elements of \( \Field L \) linearly independent over
\( \Field K \).
For every \( \lambda \),
the sum \( \sum_n \lambda_n \upsilon_n \) is invertible,
let \( I_\lambda \) denote the multiplication by its inverse.
Then \( (\Field L, \times , \upsilon_0, \upsilon_1, \upsilon_2, \dots,
I_\lambda) \) is a non-zero \( \Insane \)-algebra.

The coalgebras over \( \Insane \)
are the \( \UnitalAssociativeOperad \)-coalgebras
\( (V, \delta, \varepsilon_0) \)
endowed with other cooperations \( \varepsilon_n \)
and \( I_\lambda \) such that
\[
	I_\lambda \circ
	\left(\sum_{n \in \Naturals} \lambda_n \varepsilon_n
	\otimes \IdentityMap_V\right) \circ \delta = \IdentityMap_V.
\]
Any such coalgebra has to be trivial.
Indeed,
since \( V \) is in particular a
\( \UnitalAssociativeOperad \)\=/coalgebra,
every element \( y\in V \) generates
a finite dimensional subcoalgebra \( V_y \subset V \).
Since \( V_y \) is stable by \( \delta \),
it follows that for every \( \lambda \),
\( V_y \) is also stable by
\[
	\RightInverse_\lambda \coloneqq
	\sum_{n \in \Naturals} (\lambda_n \varepsilon_n \otimes
	\IdentityMap_V) \circ \delta.
\]
As \( \RightInverse_\lambda \) is injective on \( V \) and \( V_y \)
is finite dimensional,
we conclude that it restricts to an automorphism of \( V_y \).
As a consequence either \( V_y \IsNaturallyIsomorphicTo 0 \)
or the linear forms \( (\varepsilon_0, \varepsilon_1, \dots) \)
are independent on \( V_y \),
which would contradict the finite dimensionality of \( V_y \).
As \( V_y \IsNaturallyIsomorphicTo 0 \) only when \( y = 0 \),
we conclude that \( V \IsNaturallyIsomorphicTo 0 \).

\section{Sane operads}

\begin{definition}
	We shall say that a dg-operad \( P \) is \emph{sane} if the
	cofree \( P \)-coalgebra functor is faithful or equivalently if
	the counit of the adjunction
	\( \Forget \IsAdjointTo \Cofree \)
	\[
		\CofreeCogebraFunctor P X \longrightarrow X,
	\]
	is an epimorphism for every chain complex \( X \).
\end{definition}

Our goal is to prove that wide classes of dg-operads are sane.
The main tool for that is the propagation lemma.

\begin{lemma}[(Propagation)]%
	\label{thm:propagation}
	Let \( f \From P \to P' \) be a morphism of dg-operads.
	If \( P' \) is sane, so is~\( P \).
\end{lemma}

\begin{proof}
	Given a chain complex \( X \), 
	the counit map \( \CofreeCogebraFunctor{P'} X \to X \) factors
	through the counit map \( \CofreeCogebraFunctor P X \to X \) as
	\[
		\begin{tikzcd}
			\CofreeCogebraFunctor{P'} X
			\rar["\CofreeCogebraFunctor f"]
				& \CofreeCogebraFunctor P X \rar
					& X.
		\end{tikzcd}
	\]
	If the composite is epimorphic, so is the second map.
\end{proof}

Let us say that a dg-operad \( P \) is \emph{cofibrant} if for every
given morphism \( Q \to R \) of dg-operads that is both a degree-wise
epimorphism and a quasi-isomorphism, every map \( P \to R \) admits a
lift
\[
	\begin{tikzcd}
			& Q
			\arrow[d, two heads]
			\ar[d, phantom, shift left=1.5, near start,
			"\rotatebox{90}{\( \!\!\!\!\sim \)}"] \\
		P \ar[ru, dotted]
		\arrow[r, "", swap]
			& R.
	\end{tikzcd}
\]
In particular, a cofibrant dg-operad must lift against
\( \Interval \to 0 \), where \( \Interval \) is the
reduced cellular model of the interval
\[
\begin{tikzcd}
	{\cdots}
	\arrow[r]
	& 0
	\arrow[r]
	& \CosmosUnit
	\arrow[r,"\IdentityMap"]
	& \CosmosUnit
	\arrow[r] 
	& 0
	\arrow[r]
	& \cdots
\end{tikzcd}
\]
equipped with its canonical dg-algebra structure.

For a chain complex \( X \), let \( X\Dual \coloneqq \InternalHom X
\CosmosUnit \).
Recall that for every triple \( A, B, C \) of object in
a symmetric closed monoidal category, one has a canonical morphism
\( [A, B] \otimes C \to [A, B \otimes C] \) obtained via adjunctions.
One says that \( X \) is \emph{dualisable} when the canonical map
\[
	X \otimes X\Dual \longrightarrow \InternalHom X X
\]
is an isomorphism. As an example, a
chain complex of vector spaces is dualisable if and only if it is
(bounded and) finite dimensional.
When \( X \) is dualisable, there is given
a canonical unit map \( \CosmosUnit \to X \otimes X\Dual \) and one has
the familiar
\[
	{(X \otimes X)}\Dual
	\IsNaturallyIsomorphicTo X\Dual \otimes X\Dual
	\qand {\left(X\Dual\right)}\Dual \IsNaturallyIsomorphicTo X.
\]

\begin{theorem}[(Sanity check)]%
	\label{thm:sanity_check_one}
	A dg-operad \( P \) is sane when either
	\begin{itemize}
		\item
			\( P \) admits an augmentation
			\( P \to \CosmosUnit \);
		\item
			\( P \) is cofibrant;
		\item
			or \( P(0) \) is dualisable
			and the unit \( \CosmosUnit \to P \)
			is a degree-wise split monomorphism.
	\end{itemize}
\end{theorem}

\begin{remark}
	We let to the reader the treat of proving that for operads in
	\( \CategoryOfVectorSpacesOver {\Field K} \)
	the situation is binary:
	given a \( \Field K \)-linear operad \( P \),
	either all \( P \)\=/coalgebras are zero or for
	every vector space \( X \),
	the counit map \( \CofreeCogebraFunctor P X \to X \)
	is surjective.
\end{remark}

The first two cases follow by propagation since both the unit operad and
\( \Interval \) are sane. We shall focus on the last case.
Let us recall that a dg-operad \( P \) is \emph{reduced} when
\( P(0) = 0 \). Every dg\=/operad \( P \) admits a maximal reduced
suboperad, that we shall denote by \( \overline P \).

\begin{lemma}
	If \( P \) is a reduced dg-operad and the unit
	\( \CosmosUnit \to P \) is a degree-wise split monomorphism,
	then \( P \) is sane.
\end{lemma}

\begin{proof}
	Since \( P \) is reduced, the inclusion
	\( P(1) \subset P \) splits in the category of
	dg-operads. By propagation one is sane whenever the other is.
	Since \( P(1) \) is concentrated in arity \( 1 \),
	for every chain complex \( X \),
	the cofree \( P(1) \)-coalgebra generated by \( X \) is given by
	\( \InternalHom {P(1)} X \) with coalgebra structure
	\[
		\InternalHom {P(1)} X \longrightarrow
		\InternalHom {P(1)} {\InternalHom {P(1)} X}
		\IsNaturallyIsomorphicTo
		\InternalHom {P(1) \otimes P(1)} X
	\]
	induced by the multiplication map \( P(1) \otimes P(1) \to
	P(1) \).
	The assumption that \( \CosmosUnit \to P \) is a
	degree-wise split monomorphism ensures that for every \( X \),
	the counit \( \InternalHom {P(1)} X\to X \)
	is a degree-wise split epimorphism, so \( P(1) \) is sane.
\end{proof}

We are thus left to prove that in the case where
\( \overline P \) is sane and \( P(0) \) is dualisable, \( P \)
is sane. For this we shall investigate the adjunction between
coalgebras over \( P \) and coalgebras over \( \overline P \).

For every dg-operad, the chain complex \( P(0) \) is the
initial \( P \)-algebra. When \( P(0) \) is dualisable, one has
a canonical isomorphism of operads
\[
	\EndomorphismOperad{P(0)}
	\IsNaturallyIsomorphicTo \CoendomorphismOperad{P(0)\Dual}.
\]
The \( P \)-algebra structure of \( P(0) \) composed with the
above isomorphism
\[
	 P \longrightarrow \EndomorphismOperad{P(0)}
	 \IsNaturallyIsomorphicTo \CoendomorphismOperad{P(0)\Dual} 
\]
gives a \( P \)-coalgebra structure on \( P(0)\Dual \).
This is the terminal \( P \)-coalgebra.

\begin{definition}
	A pointed \( P \)-coalgebra is the data of a
	\( P \)\=/coalgebra \( V \) together with a morphism of
	\( P \)-coalgebras \( P(0)\Dual \to V \), called a pointing.
	A morphism of pointed
	\( P \)-coalgebras is one that respects the pointing maps.
	Let \( \CategoryOfPointedCogebrasOver P \)
	denote the category of pointed \( P \)-coalgebras.
\end{definition}

\begin{lemma}
	The category of pointed \( P \)-coalgebras
	is equivalent to the category of \( \overline P \)-coalgebras.
\end{lemma}

\begin{proof}
	For any pointed \( P \)-coalgebra \( V \),
	the cokernel \( \overline V \) of the
	pointing map \( P(0)\Dual \to V \) admits a canonical
	\( \overline P \)-coalgebra structure.
	
	Conversely,
	given a \( \overline P \)-coalgebra \( \overline V \),
	the chain complex \( \overline V \oplus P(0)\Dual \)
	can be endowed with a canonical \( P \)-coalgebra structure.
	Its restriction to \( P(0)\Dual \) is simply
	the \( P \)-coalgebra structure of \( P(0)\Dual \).
	We now describe its restriction to \( \overline V \).
	Let \( n \geq 0 \),
	the data of an \( \SymmetricGroup n \)\=/equivariant map
	\( P(n) \otimes \overline V \to (\overline V \oplus
	P(0)\Dual)^{\otimes n} \) is equivalent to the data of
	\( \SymmetricGroup p \times \SymmetricGroup q \)\=/
	equivariant maps
	\( P(n) \otimes \overline V \to
	\overline V{}^{\otimes p}
	\otimes (P(0)\Dual)^{\otimes q} \)
	for each decomposition \( p+q=n \).
	
	This structure map is 
	zero if \( p = 0 \), otherwise it is given by the composition
	\[
		\begin{tikzcd}
			\overline V \otimes P(n)
			\arrow[d,
			"\text{ unit of dualisable }{P(0)}^{\otimes q}"]
			\\
			\overline V \otimes P (n) \otimes
			{P(0)}^{\otimes q}
			\otimes {(P(0)\Dual)}^{\otimes q}
			\arrow[d,
			"\, P(n) \otimes \CosmosUnit^{\otimes p}
			\otimes {P(0)}^{\otimes q} \to P(p)"]
			\\
			\overline V \otimes P (p) 
			\otimes {(P(0)\Dual)}^{\otimes q}
			\arrow[d,
			"\text{ coalgebra structure of } \overline V"]
			\\
			\overline V^{\otimes p}
			\otimes {(P(0)\Dual)}^{\otimes q}
		\end{tikzcd}
	\]
	A straightforward checking shows this actually defines a
	\( P \)-coalgebra structure and that
	this construction \( \overline V
	\mapsto V = \overline V \oplus P(0)\Dual \) is inverse to the
	previous construction \( V \mapsto \overline V \).
\end{proof}

The adjunction relating \( P \)-coalgebras
to \( \overline P \)-coalgebras coming from the inclusion
\( \overline P \subset P \) is actually the following composite
adjunction
\[
	\begin{tikzcd}[column sep = 70]
		\CategoryOfCogebrasOver P
		\arrow[r, shift left=2, "W \mapsto W\oplus P(0)\Dual"]
			& \CategoryOfPointedCogebrasOver P
			\arrow[l, shift left=2,"\Forget"]
			\arrow[r, shift left=2, "V \mapsto \overline V"]
				& \CategoryOfCogebrasOver{\overline P} .
				\arrow[l,
				shift left=2,
				"\overline V
				\mapsto \overline V\oplus P(0)\Dual"]
	\end{tikzcd}
\]

As a consequence, for every chain complex \( X \) the counit map
\( \CofreeCogebraFunctor P X \to X \) is the composition
\[
	\begin{tikzcd}
		\CofreeCogebraFunctor P X
		= \CofreeCogebraFunctor{\overline P} X\oplus P(0)\Dual 
		\rar[two heads]
			& \CofreeCogebraFunctor{\overline P} X
			\rar[two heads]
				& X,
	\end{tikzcd}
\]
which implies that \( P \) is sane.


\section*{Acknowledgements}

The authors would like to thank both Utrecht University and the IBS
Center for Geometry and Physics which provided financial support for
their long term research project on the homotopy theory of linear
coalgebras. In particular they allowed the authors
to meet regularly either in Korea or in the Netherlands.
The first author was supported by the \texttt{NWO};
the second author was supported by \texttt{IBS-R003-D1}.

The authors would like to thank Gabriel C. Drummond-Cole for
helpful discussions and comments.


\bibliography{ms.bbl}

\tocconfig

\end{document}